\documentclass[12pt,a4paper]{amsart}
\usepackage{upgreek,subfig}
\usepackage{amssymb,amsfonts,amsmath,caption,floatrow}
\usepackage{xcolor,graphicx}
\usepackage[T1]{fontenc}
\usepackage{hyperref}
\newtheorem{thm}{Theorem}[section]\newtheorem{lem}[thm]{Lemma}
\theoremstyle{definition}\newtheorem{defn}[thm]{Definition}
\theoremstyle{remark}
\numberwithin{equation}{section}
\newcommand{\abs}[1]{\left\vert #1\right\vert}%

\newcommand{\myp}[2][0cm]{\mathopen{}\left(#2\parbox[h][#1]{0cm}{}\right)}
\newcommand{\set}[1]{\left\{ #1\right\}}%
\newcommand{\e}{\mathrm{e}}%
\renewcommand{\i}{\mathrm{i}}%
\newcommand{\q}{\mathfrak{q}_s}%
\newcommand{\BigO}[1]{\ensuremath{\operatorname{O}\bigl(#1\bigr)}}

\begin{document}

\title[Coefficients problems for families of holomorphic functions...  $\ldots$]
{Coefficients problems for families of holomorphic functions related to hyperbola}

\author[S. Kanas, V. S. Masih and A. Ebadian]{S. Kanas$^{1}$, V. S. Masih$^{2}$ and A. Ebadian$^3$}
\address{$^{1}$ University of Rzeszow, Al. Rejtana 16c, PL-35-959 Rzesz\'{o}w, Poland }
\email{skanas@ur.edu.pl}

\address{$^{2}$ Department of Mathematics, Payame Noor University, Tehran, Iran}
\email{masihvali@gmail.com; v\_soltani@pnu.ac.ir}

\address{$^{3}$ Department of Mathematics, Faculty of Science, Urmia University, Urmia, Iran}
\email{ebadian.ali@gmail.com}

\maketitle

\begin{abstract}
We consider a family of analytic and normalized functions that are related to the domains $\mathbb{H}(s)$, with a right branch of a hyperbolas $H(s)$ as a boundary. The hyperbola $H(s)$ is given by the relation $\frac{1}{\rho}=\left( 2\cos\frac{\varphi}{s}\right)^s\quad (0<s\le 1,\ |\varphi|<(\pi s)/2$). We mainly study a coefficient problem of the families of functions for which $zf'/f$ or $1+zf''/f'$ map the unit disk onto a subset of $\mathbb{H}(s)$.  We find coefficients bounds, solve Fekete-Szeg\"{o} problem and estimate the Hankel determinant.
\end{abstract}

\renewcommand{\thefootnote}{} \footnote{2010 \emph{Mathematics Subject Classification}: \textrm{30C45, 30C80}}
\footnote{\emph{Key words and phrases}: {univalent functions, subordination, starlike and convex functions, domain related to hyperbola, conic sections}}
\footnote{\textsuperscript{1} Corresponding author} \vspace{-2em}
% --------------------------------------------------------------------------------------------
\section{Introduction and definition}\label{intro}
Geometric interpretation of some properties of functions represents one of the major goal in geometric function theory of one complex variable. In this theory the description of geometries in succinct mathematical terms, establishing close links between certain prescribed and analytically expressed properties are the most desirable.
A specific relations occur with the families of domains contained in a right halplane, where the rigour of analytic reasoning one closely blends with the geometric intuition. The halfplanes, circular and angular domains, and domains bounded by conic sections have been popular and investigated so far, see for example \cite{BK, Dur, Goo, MM}, and also \cite{KW, KW1, Rob, Sta}. The relations between analytic functions and those domains relay on the inclusions of the image of some analytic expressions of the unit disk $\mathbb{D}$ in those domains.
%===========================================================================================

Very recently a new subfamily of domains contained in a right half plane and  related to a hyperbola
$$H(s)=\left\{\rho\e^{\i \varphi}:\rho=\frac{1}{\myp{2\cos\frac{\varphi}{s}}^s},\quad -\frac{\pi s}{2}<\varphi<\frac{\pi s}{2}\right\}$$
was defined \cite{KME}.  It has been assumed that $0<s\le 1$. Hyperbola  $H(s)$ intersects the real axis at $(u,0)=(2^{-s},0)$ and has a slope angle to the real axis equal $(\pi s)/2$, see Fig.1. In a limiting case $s=1$ the hyperbola becomes a line that intersects a real axis at $(u,0)=(2^{-1},0)$. The detailed description of geometric behavior of $H(s)$ and connections with the other curves of the right halfplane was given in \cite{KME}.

Let $\mathbb{H}(s)$ be a domain with $H(s)$ as a boundary. $\mathbb{H}(s)$ is symmetric with the real axis, and starlike with respect to $1\in \mathbb{H}(s)$.  The family  $\{\mathbb{H}(s), 0< s \le 1\}$  constitutes a new subfamily of domains contained in a right halfplane which  is not reduced to any family of domains considered until now. Even the hyperbolas that occurred in the consideration of conic sections $\partial\Omega_k\  (0\le k < \infty)$ (see \cite{KW, KW1}) for no choice of parameter $k$ and $s$  reduce to $\mathbb{H}(s)$.
%===========================================================================================
\begin{figure}[h]
\begin{center}
\qquad\includegraphics[scale=0.5]{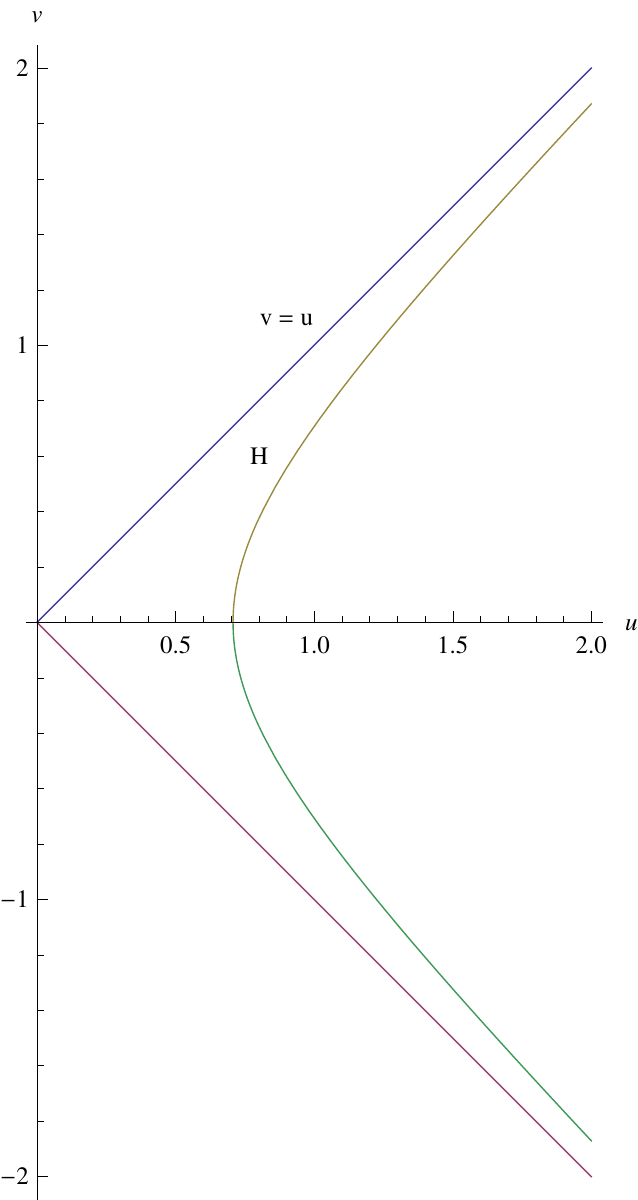}\qquad\qquad
\end{center}\end{figure}

\begin{center}Fig. 1. The hyperbola $H(s)$ for $s=1/2$.\end{center}
%=============================================================================================
By $\mathcal{H}$ we will denote the class of functions  $f$ \emph{holomorphic} in the open unit disc  $\mathbb{D}=\left\{z\in \mathbb{C}:\ |z|<1\right\}$ of the complex plane $\mathbb{C}$, and with the power series
\begin{equation}\label{A}
f(z)=z+\sum_{n=2}^{\infty}a_nz^n \qquad \myp{z \in \mathbb{D}},
\end{equation}
and by $\mathcal{S}$  its the most important subclass consisting of univalent functions.

A comprehensive geometric characteristics of members of $\mathcal{S}$ is still not known. The definitely much  known are functions from various subclasses of $\mathcal{S}$, such as $\mathcal{ST}$, $\ \mathcal{CV}$, i.e. functions  which are starlike with respect to the origin and convex, respectively.

For further considerations we denote by $\mathcal{P}$  the class of analytic functions that maps the origin to the point $1$ and has a positive real part. The class $\mathcal{P}$ has the significant relations with the majority of subclasses of $\mathcal{S}$, and is known as the Carath\'{e}odory class.
We have
$$\mathcal{P}=\{p(z)= 1+p_1z+p_2 z^2 +\cdots,\ \Re\, p(z) >0, z\in \mathbb{D}\}.$$
Now, let us denote
$$\q(z):=\frac{1}{\myp{1-z}^s}= \e^{-s\ln(1-z)}\qquad\myp{0<s\le1, \,z\in \mathbb{D}},$$
where the branch of logarithm is determined by $\q(0)=1$. The function $\q$  maps the unit disk onto a domain $\mathbb{H}(s)$ and play a role of extremal function in a class $\mathcal{P}(\q)=\{p\in \mathcal{P}: p\prec \q\}$.
The power series of $\q$ has the form
\begin{align}\label{zq}
\q(z)&=1+\sum_{n=1}^{\infty}q_nz^n=1+\sum_{n=1}^{\infty}\frac{s\myp{s+1}\cdots\myp{s+n-1}}{n!} z^n=1+\sum_{n=1}^{\infty}\frac{(s)_n}{n!} z^n,
\end{align}
where $(a)_n$ is known as the Pochhamer symbol; $(a)_n=\Gamma(a+n)/\Gamma(a)$.
The necessary and sufficient condition which describes the members of $\mathcal{P}(\q)$ was given in \cite{KME}.
\begin{thm}[\cite{KME}]\label{main0} Let $p\in \mathcal{P}$. A function $p\in \mathcal{P}(\q)$ if and only if
\begin{equation*}
\abs{p(z)^{1/s}-1} <\abs{p(z)}^{1/s}  \quad\myp{z\in \mathbb{D}}.
\end{equation*}\end{thm}

%============================================================================================================
\begin{defn}[\cite{KME}]\label{def1}
By $\mathcal{ST}_{hpl}(s)$ we denote the subfamily of $\mathcal{S}$ consisting of the functions $f$, satisfying the condition
\begin{equation*}
\frac{zf'(z)}{f(z)}\prec \q(z) \quad \myp{z\in \mathbb{D}},
\end{equation*}
and by $\mathcal{CV}_{hpl}(s)$ the subclass of univalent functions $f$ such that
\begin{equation*}
1+\frac{zf''(z)}{f'(z)}\prec \q(z) \quad \myp{z\in \mathbb{D}},
\end{equation*}
where $\prec$ denotes the symbol of subordination.
\end{defn}
Let $\Phi_{s,n}\in \mathcal{ST}_{hpl}(s)$ be given by
\[
\frac{z\Phi'_{s,n}(z)}{\Phi_{s,n}(z)} = \frac{1}{\myp{1-z^n}^s} \quad \myp{z\in \mathbb{D}, \, n=1, 2,  \ldots}.
\]
Then, the functions $\Phi_{s,n}(z)$ is of the form
\begin{multline}\label{F}
\Phi_{s,n}(z)=z\exp\myp{ \int_{0}^{z} \frac{\mathfrak{q}_{s}(t^n)-1}{t}\,
   \mathrm{dt}}=z+\frac{s}{n}z^{n+1}+\frac{(n+2)s^2+ns}{4n^2}z^{2n+1}\\
   +\frac{4n^2s+(9n+6n^2)s^2+(2n^2+9n+6)s^3}{36n^3}z^{3n+1}\cdots.
\end{multline}
Especially for $n=1$,
\begin{equation}\label{F1}
\Phi_{s}(z):=\Phi_{s,1}(z)=z+s\,z^{2}+\frac{3s^2+s}{4}z^{3}+\frac{4s+15 s^2+17s^3}{36}z^{4}+\cdots.
\end{equation}
Also, let $K_{s,n}\in \mathcal{CV}_{hpl}(s)$ be given by
\[
1+\frac{zK''_{s,n}(z)}{K'_{s,n}(z)} = \frac{1}{\myp{1-z^n}^s} \quad \myp{z\in \mathbb{D}, \, n=1, 2,  \ldots}.
\]
From the above, we find that
\begin{align}\label{K}
K_{s,n}(z)&=\int_{0}^{z}\set{\exp \int_{0}^{w} \frac{\mathfrak{q}_{s}(t^n)-1}{t}\,
   \mathrm{d}t}\mathrm{d}w= z+\frac{s}{n(n+1)}z^{n+1}\\
  & +\frac{(n+2)s^2+ns}{4n^2(2n+1)}z^{2n+1}
  +\frac{4n^2s+(9n+6n^2)s^2+(2n^2+9n+6)s^3}{36n^3(3n+1)}z^{3n+1}\cdots.\notag
\end{align}
Especially for $n=1$,
\begin{equation}\label{K1}
K_{s}(z):=K_{s,1}(z)= z+\frac{s}{2}z^{2} +\frac{3s^2+s}{12}z^{3}+ \frac{4s+15s^2+17s^3}{144}z^4+\cdots.
\end{equation}

%===========================================================================================
\section{Preliminary results}
%===========================================================================================
In order to achieve our aim  we recall some definitions  and lemmas that will be useful in the next part of our paper.
%===========================================================================================

By $\mathcal{B}$ we denote the class of analytic selfmappings of the unit disk, that maps the origin onto the origin, i.e.
\begin{equation*}
\mathcal{B}=\left\{\omega(z)\in \mathcal{H},\ \omega(0)=0,\ |\omega(z)|<1,\ z\in\mathbb{D}\right\}.\end{equation*}
The class $\mathcal{B}$ is known as the class of \textit{Schwarz functions}.

\begin{lem}[\cite{KAZ}]\label{lem1}
 For the function $\omega\in \mathcal{B}$ and $\omega(z)=w_1z+w_2z^2+\cdots$ it holds
 \begin{align*}
 w_2&=\xi (1-w_1^2),\\
 w_3&=(1-w_1^2)(1-|\xi|^2)\zeta-w_1(1-w_1^2)\xi^2,
 \end{align*}
 for some complex number $\xi$, $\zeta$ with $\abs{\xi}\leq 1$ and $\abs{\zeta}\leq 1$.
\end{lem}
%===========================================================================================
\begin{lem}[\cite{ARS}]\label{lemma1}
If $\omega(z)=\sum_{n=1}^{\infty}w_nz^n\in \mathcal{B}$, then for real numbers $t$ it holds
 \begin{equation*}
    \abs{w_2-t w_1^2}
    \leq  \left\{
\begin{array}{ll}
	-t\quad &\textit{for}\quad t \le-1,\\[0.3em]
	1\quad &\textit{for}\quad -1\leq t \leq 1, \\[0.3em]
	t	\quad &\textit{for}\quad t \ge 1.
\end{array}\right.
 \end{equation*}
When $t<1$ or $t>1$, the equality holds if and only if $\omega(z)=z$ or one of its rotations. If $-1<t<1$, then equality
holds if and only if $\omega(z)=z^2$ or one of its rotations. Equality holds for $t=-1$ if and only if $\omega(z)=z\frac{x+z}{1+xz}$ $\ \myp{0\le x\le 1}$ or one of its rotations while for $t = 1$, equality holds if and only $\omega(z)=-z\frac{x+z}{1+xz}\ \myp{0\le x\le 1}$ or one of its rotations.
\end{lem}
%===========================================================================================

We  say that $f$ is \textit{subordinate} to $F$  in $\mathbb{D}$, written
$f(z)\prec F(z)$  (or $f\prec F$), if there exists a function $\omega\in \mathcal{B}$ such that $f(z)=F(\omega(z)) \ (z\in \mathbb{D})$ (see, for example  \cite{Dur}). If  $F$ is univalent in $\mathbb{D}$, then $f\prec F$ holds  if and only, if $f(0)=F(0)$ and $f(\mathbb{D})\subset F(\mathbb{D})$.

\begin{thm}[\cite{Rog}]\label{Rog}
Let  $g(z)=\sum_{n=1}^{\infty}B_nz^n$ be analytic and convex univalent in $\mathbb{D}$. If $h(z)=\sum_{n=1}^{\infty}A_nz^n$ is analytic
in $\mathbb{D}$ and satisfies the  subordination $h\prec g$ in $\ \mathbb{D}$, then
$$\abs{A_n}\leq \abs{B_1} \qquad \myp{n=1, 2, \dots}.$$
\end{thm}

\begin{defn}
Let $q,n \ge 1$. By $H_q(n)$ we denote the $q^{th}$ \textit{Hankel determinant} of the form
\begin{equation*}
H_q(n) =
\begin{vmatrix}
a_n & a_{n+1} & \dots & a_{n+q-1} \\
a_{n+1} & a_{n+2} &\dots & a_{n+q} \\
\vdots & \vdots  & \ddots &\vdots \\
a_{n+q-1} & a_{n+q}  & \dots & a_{n+2q-2}
\end{vmatrix}\qquad \myp{a_1=1},
\end{equation*}
where $a_k\ (k=1,2,...)$ are the coefficients  of the Taylor series expansion of function $f$ of the form \eqref{A}.
\end{defn}
\noindent $H_q(n)$ was defined by  Pommerenke \cite{Pom1, Pom2}, and for fixed $q,n$ the bounds of $|H_q(n)|$ have been studied for several subfamilies of univalent functions. The Hankel determinants $H_2(1)=a_3-a_2^2$ and $H_2(2)=a_2 a_4 - a_3^2$,  are well-known as \emph{Fekete-Szeg\"{o}} and \emph{second Hankel determinant functionals}, respectively. Further, Fekete and Szeg\"{o}  introduced the generalized functional
$a_3-\lambda a^2_2$, where $\lambda$ is some real number.
We will give the sharp upper bound for the some Hankel determinants in the subclasses of univalent functions related to hyperbola $H(s)$.
%===========================================================================================

%=======================================================================================
\section{Coefficient bounds}
%=======================================================================================
This section is devoted to the general problem of coefficients in a class $\mathcal{ST}_{hpl}(s)$ and $\mathcal{CV}_{hpl}(s)$. We present the estimates of coefficients for elements of $\mathcal{ST}_{hpl}(s)$ and  bounds of logarithmic coefficients, but we also solve  the Fekete-Szeg\"{o} problem and a very popular the Hankel determinant $|H_2(2)|$  problem.
%======================================================================================================
First we investigate the sharp order of growth for the coefficients of functions in $\mathcal{ST}_{hpl}(s)$ and $\mathcal{CV}_{hpl}(s)$ with $0<s<1/2$.
\begin{thm}
Assume that the function $f$ of the form \eqref{A} belongs to the class $\mathcal{ST}_{hpl}(s)$ $\myp{\mathcal{CV}_{hpl}(s), respectively}$ with $0<s<1/2$. Then $\abs{a_n}=\BigO{1/n}\ (\abs{a_n}=\BigO{1/n^2}$, resp.) for $n=1, 2, 3, \dots$.
\end{thm}
\begin{proof}
From \cite[pp. 8]{MR}, the functions $1/(1-z)^s$ for $0<s<1/2$ is in the class $H^2$, the Hardy class of analytic functions in $\mathbb{D}$. Now, from the results in \cite[Theorem 4]{MMin} we conclude the theorem.
\end{proof}
%=======================================================================================================
Using properies of the functions $p\in \mathcal{P}(\q)$ we can obtain more satisfactory bounds for $|a_n|$.
%=======================================================================================================
\begin{thm}
Assume that the function $f$ of the form \eqref{A} belongs to the class $\mathcal{ST}_{hpl}(s)$. Then
\begin{equation}\label{coeff19}
 \abs{a_n}\leq  \frac{(s)_{n-1}}{(n-1)!}\qquad \myp{n=2, 3, 4, \dots}.
\end{equation}

\end{thm}
\begin{proof}
Let $f\in \mathcal{ST}_{hpl}(s)$. Then,  from Definition \ref{def1}, there exists $p\in \mathcal{P}(\q)$ such that
\begin{equation}\label{eqthC01} p(z)=\frac{zf'(z)}{f(z)}=1+p_1z+p_2z^2+\dots \prec \q(z).\end{equation}
From the above relation we find the following relation between the coefficients of $f$ and $p$ as follows
\begin{equation}\label{eqthC3}
a_2=p_1 \quad \text{and} \quad na_n=p_{n-1}+a_2p_{n-2}+\cdots+a_{n-1}p_1+a_n \quad \myp{n=3, 4,  \dots}.
\end{equation}
Applying Theorem \ref{Rog} we deduce that $|p_n|\leq s \ \myp{n=1,2, 3, \dots}$,  then $|a_2|=|p_1|\le s$, moreover by \eqref{eqthC3}
\begin{equation}\label{eqthC4}
\abs{a_n} \leq \frac{s}{n-1}\myp{1+ \sum_{k=2}^{n-1}\abs{a_k}} \qquad \myp{n=3, 4,  \dots}.
\end{equation}
We next proceed by induction. Then, we have to prove that
\begin{equation}\label{eqthC4}
|a_{n+1}|\le \frac{(s)_n}{n!} \quad \myp{n=2, 3\ldots},
\end{equation}
assumming that the inequality $|a_n|\le (s)_{n-1}/(n-1)!$ holds for $n\ge 2$. Applying \eqref{eqthC4} we observe that
$$\begin{array}{rcl}|a_3|&\le & \dfrac{s}{2}(1+|a_2|)\ \le\  \dfrac{s}{2}(1+s)=\dfrac{s(s+1)}{2!},\\
|a_4|&\le& \dfrac{s}{3}\left(1+s+\dfrac{s(s+1)}{2!}\right)\ =\ \dfrac{s(s+1)(s+2)}{3!},\end{array}$$
and finally
$$\begin{array}{rcl}|a_{n+1}|&\le&\dfrac{s}{n}\left(1+s+\dfrac{s(s+1)}{2!}+\dfrac{s(s+1)(s+2)}{3!}+\cdots +\dfrac{s(s+1)\cdot ... \cdot (s+n-2)}{(n-1)!}\right)\\
&=& \dfrac{s(s+1)(s+2)\cdot ... \cdot (s+n-2)}{(n-1)!}.\end{array}$$
Therefore
$$|a_{n+1}|\le  \dfrac{(s)_n}{n!}, $$
and the assertion holds. Applying the Principle of Mathematical Induction, the formula \eqref{coeff19} is proved.\end{proof}
%==========================================================================================
\begin{thm}\label{coeff11}
If $f$  of the form \eqref{A}  belongs to $\mathcal{CV}_{hpl}(s)$, then
\begin{equation}\label{coeff1}
 \abs{a_n}\leq  \frac{(s)_{n-1}}{n!}\qquad \myp{n=2, 3, 4, \dots}.
\end{equation}

\end{thm}
\begin{proof} Let $f\in \mathcal{CV}_{hpl}(s)$. Using the fundamental Alexander relation between $\mathcal{CV}_{hpl}(s)$ and $\mathcal{ST}_{hpl}(s)$, we have  $zf'\in \mathcal{ST}_{hpl}(s)$, which gives that $na_n$ satisfies  \eqref{coeff19}, and then the assertion follows.\end{proof}
%===========================================================================================

\begin{thm} \label{th:feket1}
Let $f\in \mathcal{ST}_{hpl}(s)$ be given by \eqref{A}. Then
\begin{equation}\label{Hankel1}
|H_2(2)|=\abs{a_2a_4-a_3^2}\leq \frac{s^2}{4}.\end{equation}
The inequality is sharp; equality holds if $f$ is a rotation of  $\Phi_{s,2}$, given by \eqref{F}.
\end{thm}
\begin{proof}
Let the function $f$ of the form \eqref{A} be in the class $\mathcal{ST}_{hpl}(s)$. Then there exists a function $\omega\in \mathcal{B}$, \ $\omega(z) = w_1 z+w_2z^2+\cdots$,  such that
\begin{equation}\label{eq:th12}
\frac{zf'(z)}{f(z)}=\frac{1}{\myp{1-\omega(z)}^s} \qquad \myp{z\in \mathbb{D}}.
\end{equation}
We have
\begin{equation*}
\frac{zf'(z)}{f(z)}=1+a_2z+\myp{2a_3-a_2^2}z^2+\myp{a_2^3-3a_2a_3+3a_4}z^3+\cdots,\end{equation*}
and
\begin{equation*}\begin{array}{rcl}
\dfrac{1}{\myp{1-\omega(z)}^s}=1+sw_1z&+&\left(sw_2+\dfrac{s(s+1)}{2}w_1^2\right)z^2\\
&+&\left(sw_3+s(s+1)w_1w_2+\dfrac{s(s+1)(s+2)}{6}w_1^3\right)z^3+\cdots.\end{array}
\end{equation*}
Comparing the coefficients of $z$, $z^2$ and $z^3$  of both sides of the series expansion of \eqref{eq:th12}, we obtain
\begin{equation}\label{coffiecient}
a_2 =s w_1,
a_3 = \dfrac{s}{2}\myp{w_2+\dfrac{3s+1}{2}w_1^2},
a_4 = \dfrac{s}{3}\myp{w_3+\dfrac{5s+2}{2}w_1w_2+\dfrac{17s^2+15s+4}{12}w_1^3}.
\end{equation}
Hence, we have
\begin{equation*}
|H_2(2)|=\abs{a_2a_4-a_3^2}=\frac{s^2}{12}\abs{4w_1w_3-3w_2^2+(s+1)w_1^2w_2+\frac{7+6s-13s^2}{12}w_1^4}.
\end{equation*}
Using Lemma \ref{lem1}, we write the expression $w_2$ and $w_3$ in terms of $w_1$. Since the class and $H_2(2)$ are invariant under the rotation, then without loss of generality  we can assume that $x=w_1$, with $0\leq x\leq 1$. Thus
\begin{multline*}
|H_2(2)|=\frac{s^2}{12}\biggl|4x\myp{1-x^2}\myp{1-|y|^2}\zeta-4x^2\myp{1-x^2}y^2-3y^2\myp{1-x^2}^2\\+(s+1)x^2y\myp{1-x^2}
+\frac{7+6s-13s^2}{12}x^4\biggr|
\end{multline*}
where $y , \zeta$ are complex number with $|y|\le 1$ and $|\zeta|\le 1$. Using the triangle inequality, we obtain
\begin{multline*}
|H_2(2)|\leq{}\frac{s^2}{12}\biggl\{4x\myp{1-x^2}\myp{1-|y|^2}+4x^2\myp{1-x^2}|y|^2+3|y|^2\myp{1-x^2}^2\\+(s+1)x^2|y|\myp{1-x^2}
+\frac{7+6s-13s^2}{12}x^4\biggr\}.
\end{multline*}
Now, we begin by analyzing a behavior of $h(|y|)$, that represent an expression in a curly brackets of the right hand side of the above inequality. Since \[
h'(|y|)=2\myp{x-1}^2\myp{x+1}\myp{3-x}|y|+(s+1)x^2\myp{1-x^2}\geq 0
\]
  the function $h(|y|)$ is increasing  on the interval  $[0,1]$ so that $h(|y|)$ attains its  greatest value at $|y|=1$, i.e. $h(|y|)\leq h(1)$. Consequently
\begin{align*}
|H_2(2)|\leq \frac{s^2}{12}\set{-\frac{13s^2+6s+17}{12}x^4+(s-1)x^2+3}\leq \frac{s^2}{4}.
\end{align*}
The function $\Phi_{s,2}$ with its rotation, where $\Phi_{s,2}$ is given by \eqref{F}, shows that the bound $s^2/4$ is sharp.
\end{proof}
%====================================================================================================
\begin{thm} Let $h\in \mathcal{CV}_{hpl}(s)$, and $h(z) = z+d_2z^2+d_3z^3+\cdots$. Then
\begin{equation}\label{Hankel2}
|H_2(2)|=\abs{d_2d_4-d_3^2}\leq\dfrac{s^2}{36}\left(1+\frac{9s^2}{8(s^2+3s+6)} \right). \end{equation}
\end{thm}
\begin{proof}
Let the function $h(z) = z+d_2z^2+d_3z^3+\cdots$ be in the class $\mathcal{CV}_{hpl}(s)$. Then there exists a function $\omega\in \mathcal{B}$, \ $\omega(z) = w_1 z+w_2z^2+\cdots$,  such that
\begin{equation}\label{eq:coffecition3}
1+\dfrac{zh''(z)}{h'(z)}=\dfrac{1}{\myp{1-\omega(z)}^s} \qquad \myp{z\in \mathbb{D}}.
\end{equation}
We have
\begin{equation}\label{eq:coffecition4}
1+\dfrac{zh''(z)}{h'(z)}=1+2d_2z+\myp{6d_3-4d_2^2}z^2+\myp{12d_4-18d_2d_3+8d_2^3}z^3+\cdots
\end{equation}
  Comparing the coefficients of the both sides of the series expansion of \eqref{eq:coffecition3}, we obtain
\begin{equation}\label{coffiecientK}
\begin{array}{ll}
d_2 &=\dfrac{s w_1}{2},\\[1em]
d_3 &= \dfrac{s}{6}\myp{w_2+\dfrac{3s+1}{2}w_1^2},\\[1em]
d_4 &= \dfrac{s}{12}\myp{w_3+\dfrac{5s+2}{2}w_1w_2+\dfrac{17s^2+15s+4}{12}w_1^3}.
\end{array}
\end{equation}
Thus, we have
\begin{equation*}
\abs{d_2d_4-d_3^2}=\dfrac{s^2}{36}\abs{\dfrac{3}{2}w_1w_3-w_2^2+\dfrac{3s+2}{4}w_1^2w_2+\dfrac{2+3s-s^2}{8}w_1^4}.
\end{equation*}
Using Lemma \ref{lem1}, we write the expression $w_2$ and $w_3$ in terms of $w_1$. Since the class and $H_2(2)$ are invariant under the rotation, then without loss of generality  we can assume that $x=w_1$, with $0\leq x\leq 1$. Thus
\begin{multline*}
\abs{d_2d_4-d_3^2}=\frac{s^2}{36}\biggl|\dfrac32x\myp{1-x^2}\myp{1-|y|^2}\zeta-\dfrac32x^2\myp{1-x^2}y^2-y^2\myp{1-x^2}^2\\+\dfrac{3s+2}{4}x^2y\myp{1-x^2}
+\frac{2+3s-s^2}{8}x^4\biggr|
\end{multline*}
where $y , \zeta$ are complex number with $|y|\le 1$ and $|\zeta|\le 1$. Using the triangle inequality, we obtain
\begin{multline*}
\abs{d_2d_4-d_3^2}\le\frac{s^2}{36}\biggl|\dfrac32x\myp{1-x^2}\myp{1-|y|^2}-\dfrac32x^2\myp{1-x^2}|y|^2-|y|^2\myp{1-x^2}^2\\+\dfrac{3s+2}{4}x^2|y|\myp{1-x^2}
+\frac{2+3s-s^2}{8}x^4\biggr|.
\end{multline*}
Now, we begin by analyzing a behavior of $h(|y|)$, that represent an expression in a curly brackets of the right hand side of the above inequality. Since \[
h'(|y|)=\myp{1-x}^2\myp{x+1}\myp{2-x}|y|+\dfrac{3s+2}{4}x^2\myp{1-x^2}\geq 0
\]
  the function $h(|y|)$ is increasing  on the interval  $[0,1]$ so that $h(|y|)$ attains its  greatest value at $|y|=1$, i.e., $h(|y|)\leq h(1)$. Consequently
\begin{align*}
\abs{d_2d_4-d_3^2}\leq \frac{s^2}{36}\set{1+\dfrac{3s}{4}x^2-\frac{s^2+3s+6}{8}x^4}=:g(x).
\end{align*}
Since $g(x)$  attains its maximum at $x=(3s)/(s^2+3s+6)$, thus \eqref{Hankel2} follows. 
\end{proof}
%====================================================================================================
\begin{thm}\label{thFS}
Let $f\in \mathcal{ST}_{hpl}(s)$ be given by \eqref{A}. Then the following  sharp inequalities for Fekete-Szeg\"{o} functional hold
\[\abs{a_3-\lambda a_2^2} \leq \left\{
\begin{array}{lcl}
	s^2\myp{\lambda-\dfrac{3s+1}{4s}}&\textit{for}&\lambda \ge \dfrac{3s+3}{4s},\\[1em]
	\dfrac{s}{2}&\textit{for}& \dfrac{3s-1}{4s}\leq \lambda \leq \dfrac{3s+3}{4s}, \\[1em]
	-s^2\myp{\lambda-\dfrac{3s+1}{4s}}&\textit{for}& \lambda \le \dfrac{3s-1}{4s}.
\end{array}\right.
\]
\end{thm}
\begin{proof}
Form \eqref{coffiecient}, we have
\begin{equation*}
\abs{a_3-\lambda a_2^2}=\dfrac{s}{2}\abs{w_2-w_1^2\myp{2s\lambda-\dfrac{3s+1}{2}}}.
\end{equation*}
The assertion is proved by the application of  Lemma \ref{lemma1} with $ t=2s\lambda-(3s+1)/2$. Equality is attained in the case $\myp{3s-1}/\myp{4s}< \lambda < \myp{3s+3}/\myp{4s}$ for $f(z)=\overline{\mu} \Phi_{s,2}(\mu z)$, where $\Phi_{s,n}$ is  given by \eqref{F}, and by the function $f(z)=\overline{\mu} \Phi_{s,1}(\mu z)\ $  for $\lambda > \myp{3s+3}/\myp{4s}$ or $\lambda<\myp{3s-1}/\myp{4s} $ (here $\mu$ is an unimodular constant). For $x\in[0,1]$ we define the normalized analytic functions $f_x$ and $g_x$ by
\begin{equation}\label{fg}
\dfrac{zg'_x(z)}{g_x(z)}=\mathfrak{q}_{s}\myp{\dfrac{z(z+x)}{1+xz}}\quad \textit{and}\quad\dfrac{zf'_x(z)}{f_x(z)}=\mathfrak{q}_{s}\myp{-\dfrac{z(z+x)}{1+xz}}
\end{equation}
When $\lambda =\myp{3s+3}/\myp{4s}$  equality holds   for $f(z)=\overline{\mu} f_x(\mu z)$. If  $\lambda=\myp{3s-1}/\myp{4s}$  equality holds   for $f(z)=\overline{\mu} g_x(\mu z)$, where $|\mu|=1$.
\end{proof}

%================================================================================================

\begin{thm}\label{th.3.7}
Let $h\in \mathcal{CV}_{hpl}(s)$, and $h(z) = z+d_2z^2+d_3z^3+\cdots$. Then we have sharp inequalities
\[\abs{d_3-\lambda d_2^2} \leq   \left\{
\begin{array}{lcl}
	-\dfrac{s^2}{4}\myp{\lambda-\dfrac{3s+1}{3s}}&\textit{for}&\lambda \le \dfrac{3s-1}{3s},\\[1em]
	\dfrac{s}{6}&\textit{for}& \dfrac{3s-1}{3s}\leq \lambda \leq \dfrac{s+1}{s}, \\[1em]
	\dfrac{s^2}{4}\myp{\lambda-\dfrac{3s+1}{3s}}&\textit{for}& \lambda \ge \dfrac{s+1}{s}.
\end{array}\right.
\]
\end{thm}
\begin{proof}Reasoning along the same line as in the proof of Theorem \ref{thFS}, we get
$$|d_3-\lambda d_2^2|=\dfrac{s}{6}\left|w_2-\dfrac{3s\lambda-3s-1}{2}w_1^2\right|.$$
Setting $ \mu=\dfrac{3s\lambda-3s-1}{2}$ in Lemma \ref{lemma1} we obtain the assertion.
The inequalities are sharp for the functions
\begin{equation*}f(z)= \left\{
\begin{array}{ll}
        \overline{\mu} K_{s,2}(\mu z)\quad &\textit{for}\quad \dfrac{3s-1}{3s}< \lambda < \dfrac{s+1}{s},\\[0.5em]
       \overline{\mu} K_{s}(\mu z)\quad &\textit{for}\quad \lambda \in\myp{-\infty, \dfrac{3s-1}{3s}}\cup\myp{\dfrac{s+1}{s},\infty},\\[0.5em]
  \overline{\mu}F_x(\mu z)    \quad &\textit{for}\quad\lambda=\dfrac{s+1}{s},\\[0.5em]
   \overline{\mu}G_x(\mu z)    \quad &\textit{for}\quad\lambda = \dfrac{3s-1}{3s},
\end{array}\right.
\end{equation*}
where $K_{s,2}$ and  $K_{s}$ are given by \eqref{K} and \eqref{K1},  $\mu$ is an unimodular constant and the functions $F_x$ and $G_x\ \myp{0\le x\le 1}$ are defined by
\begin{equation}\label{fgx}
1+\dfrac{zG''_x(z)}{G'_x(z)}=\mathfrak{q}_{s}\myp{\dfrac{z(z+x)}{1+xz}}\quad \textit{and}\quad 1+\dfrac{zF''_x(z)}{F'_x(z)}=\mathfrak{q}_{s}\myp{-\dfrac{z(z+x)}{1+xz}}.
\end{equation}
\end{proof}
%====================================================================================================
Now, we find the bounds of the Fekete-Szeg\"{o} functionals of $z/f(z)$ and $f^{-1}(z)$ when $f$ is an element of the class
$\mathcal{ST}_{hpl}(s)$ or $\mathcal{CV}_{hpl}(s)$.

Let the function $F$ be defined by
\begin{equation}\label{eq41}
F(z)= \dfrac{z}{f(z)}=1+\sum_{n=1}^{\infty}b_nz^n \quad \myp{z\in\mathbb{D}},
\end{equation}
for $f\in \mathcal{S}$ given by \eqref{A}.

\begin{thm}Let $f\in \mathcal{ST}_{hpl}(s)$,  and let $F(z)=z/f(z)$ be of the form \eqref{eq41}. Then the following sharp inequalities hold
\[\abs{b_2-\lambda b_1^2} \leq \left\{
\begin{array}{lcl}
	s^2\myp{\dfrac{s-1}{4s}-\lambda}&\textit{for}& \lambda \le \dfrac{s-3}{4s},\\[1em]
	\dfrac{s}{2} &\textit{for}& \dfrac{s-3}{4s}\leq \lambda \leq \dfrac{s+1}{4s}, \\[1em]
	s^2\myp{\lambda-\dfrac{s-1}{4s}}&\textit{for}& \lambda \ge \dfrac{s+1}{4s}.
\end{array}\right.
\]
\end{thm}
\begin{proof}
For given $f\in \mathcal{ST}_{hpl}(s)$ with the power series \eqref{A} let $F(z)=\dfrac{z}{f(z)}$. An easy computation shows that
\begin{equation}\label{eq42}
F(z)=\dfrac{z}{f(z)}=1-a_2z+\myp{a_2^2-a_3}z^2+\cdots \qquad\myp{z\in \mathbb{D}}.
\end{equation}
Equating \eqref{eq41} and \eqref{eq42}, we obtain
\begin{equation}\label{coffiecient1}
b_1 =-a_2,\quad
b_2 = a_2^2-a_3,
\end{equation}
and then 
\begin{equation*}
\abs{b_2-\lambda b_1^2}=\abs{a_3-(1-\lambda)a_2^2}.
\end{equation*}
Assertion now follows by application of Theorem \ref{thFS} with $1-\lambda$.
The inequalities are sharp, the equality holds for the functions
\begin{equation*}f(z)= \left\{
\begin{array}{ll}
        \overline{\mu} \Phi_{s,2}(\mu z)\quad &\textit{for}\quad \dfrac{s-3}{4s}< \lambda < \dfrac{s+1}{4s},\\[0.5em]
       \overline{\mu} \Phi_{s}(\mu z)\quad &\textit{for}\quad \lambda \in\myp{-\infty, -\dfrac{s-3}{4s}}\cup\myp{\dfrac{s+1}{4s},\infty},\\[0.5em]
  \overline{\mu}f_x(\mu z)    \quad &\textit{for}\quad\lambda=\dfrac{s-3}{4s},\\[0.5em]
   \overline{\mu}g_x(\mu z)    \quad &\textit{for}\quad\lambda = \dfrac{s+1}{4s},
\end{array}\right.
\end{equation*}
where $\Phi_{s,2}$ and  $\Phi_{s}$ are given by \eqref{F} and \eqref{F1},  $\mu$ is an unimodular constant and the functions $f_x$ and $g_x\ \myp{0\le x\le 1}$ are given by \eqref{fg}.
\end{proof}
%===================================================================================================
\begin{thm}
Let $h\in \mathcal{CV}_{hpl}(s)$  and $h(z)=z+d_2z^2+d_3z^3+\cdots$. Also, let $z/h(z)=1+\sum_{n=1}^{\infty}h_nz^n$. Then the following sharp inequalities hold.
\[\abs{h_2-\lambda h_1^2} \leq \left\{
\begin{array}{lcl}
	\dfrac{s^2}{4}\myp{\lambda+\dfrac{1}{3s}}&\textit{for}& \lambda \ge \dfrac{1}{3s},\\[1em]
	\dfrac{s}{6} &\textit{for}& -\dfrac{1}{s}\leq \lambda \leq \dfrac{1}{3s}, \\[1em]
	-\dfrac{s^2}{4}\myp{\lambda+\dfrac{1}{3s}}&\textit{for}& \lambda \le -\dfrac{1}{s}.
\end{array}\right.
\]
\end{thm}
\begin{proof}
Let $h\in \mathcal{CV}_{hpl}(s)$. A simple computation shows that
$z/h(z)=1+\sum_{n=1}^{\infty}h_nz^n=1-d_2z^2+(d_2^2-d_3)z^3+\cdots$. Then
\[
\abs{h_2-\lambda h_1^2}=\abs{d_3-\myp{1-\lambda}d_2^2}.
\]
Assertion now follows by application of Theorem \ref{th.3.7} with $1-\lambda$. The inequalities are sharp for the functions
\begin{equation*}f(z)= \left\{
\begin{array}{ll}
        \overline{\mu} K_{s,2}(\mu z)\quad &\textit{for}\quad \dfrac{-1}{s}< \lambda < \dfrac{1}{3s},\\[0.5em]
       \overline{\mu} K_{s}(\mu z)\quad &\textit{for}\quad \lambda \in\myp{-\infty, \dfrac{-1}{s}}\cup\myp{\dfrac{1}{3s},\infty},\\[0.5em]
  \overline{\mu}F_x(\mu z)    \quad &\textit{for}\quad\lambda=-\dfrac{1}{s},\\[0.5em]
   \overline{\mu}G_x(\mu z)    \quad &\textit{for}\quad\lambda = \dfrac{1}{3s}.
\end{array}\right.
\end{equation*}
where $\mu$ is an unimodular constant, the functions $F_x$ and $G_x\ \myp{0\le x\le 1}$ are given by \eqref{fgx},  $K_{s,2}$ and  $K_{s}$ are given by \eqref{K} and \eqref{K1}.
\end{proof}
%=====================================================================================================

Now, let the function $f^{-1}$ be defined by
\begin{equation}\label{eq51}
f^{-1}(w)=w+\sum_{n=2}^{\infty}A_nw^n \qquad \myp{|w|<r_0(f)},
\end{equation}
where $r_0(f)$  is the Koebe constant of the class $\mathcal{ST}_{hpl}(s)$ (or $\mathcal{CV}_{hpl}(s)$, respectively). By the Koebe constant of a class $\mathcal{F}$ we mean the largest disk $\{|w|<d\}$ that is contained in $f(\mathbb{D})$ for every $f\in \mathcal{F}$. The problem of Koebe constant was
solved in \cite{KME}.  Then
\begin{equation*}
f^{-1}\myp{f(z)}=z \qquad \myp{z \in \mathbb{D}}\qquad  \textit{and}\qquad f\myp{f^{-1}(w)}=w \qquad \myp{|w|<r_0(f)}.
\end{equation*}
The inverse function $f^{-1}$ is given by
\begin{equation}\label{eq52}
f^{-1}(w)=w-a_2w^2+\myp{2a_2^2-a_3}w^3-\myp{5a_2^3-5a_2a_3+a_4}w^4+\cdots.
\end{equation}
%=====================================================================================================

\begin{thm}\label{th_inverse_f}
Let $f\in \mathcal{ST}_{hpl}(s)$ be of the form \eqref{A},  and let $f^{-1}(z)$ be given by  \eqref{eq51}. Then the sharp inequalities hold
\[\abs{A_3-\lambda A_2^2} \leq \left\{
\begin{array}{lcl}
	s^2\myp{\dfrac{5s-1}{4s}-\lambda} &\textit{for}& \lambda \le  \dfrac{5s-3}{4s},\\[1em]
	\dfrac{s}{2}&\textit{for}& \dfrac{5s-3}{4s}\leq \lambda \leq \dfrac{5s+1}{4s}, \\[1em]
	s^2\myp{\lambda-\dfrac{5s-1}{4s}}&\textit{for}& \lambda \ge \dfrac{5s+1}{4s}.
\end{array}\right.
\]
\end{thm}
\begin{proof}
Equations \eqref{eq51} and \eqref{eq52} give
\begin{equation}\label{coffiecient2}
A_2 =-a_2,\quad
A_3 =2a_2^2-a_3.
\end{equation}
Applying  \eqref{coffiecient} to \eqref{coffiecient2}, we obtain
\begin{equation*}
\abs{A_3-\lambda A_2^2}=\abs{a_3-(2-\lambda)a_2^2}.
\end{equation*}
Assertion now follows by application of Theorem \ref{thFS} with $2-\lambda$.
The inequalities are sharp for the functions
\begin{equation*}f(z)= \left\{
\begin{array}{ll}
        \overline{\mu} \Phi_{s,2}(\mu z)\quad &\textit{for}\quad \dfrac{5s-3}{4s}< \lambda < \dfrac{5s+1}{4s},\\[0.5em]
       \overline{\mu} \Phi_{s}(\mu z)\quad &\textit{for}\quad \lambda \in\myp{-\infty, \dfrac{5s-3}{4s}}\cup\myp{\dfrac{5s+1}{4s},\infty},\\[0.5em]
  \overline{\mu}f_x(\mu z)    \quad &\textit{for}\quad\lambda=\dfrac{5s-3}{4s},\\[0.5em]
   \overline{\mu}g_x(\mu z)    \quad &\textit{for}\quad\lambda = \dfrac{5s+1}{4s}.
\end{array}\right.
\end{equation*}
where $\Phi_{s,2}$ and  $\Phi_{s}$ are given by \eqref{F} and \eqref{F1},  $\mu$ is an unimodular constant, and $f_x$ and $g_x\ \myp{0\le x\le 1}$ are given by \eqref{fg}.
\end{proof}
%================================================================================================
\begin{thm}
Let $h\in \mathcal{CV}_{hpl}(s)$ with $h(z)=z+d_2z^2+d_3z^3+\cdots$. If $h^{-1}(w)=w+\sum_{n=2}^{\infty}B_nw^n$, then the sharp inequalities hold
\[\abs{B_3-\lambda B_2^2} \leq \left\{
\begin{array}{lcl}
	\dfrac{s^2}{4}\myp{\lambda-\dfrac{3s-1}{3s}} &\textit{for}& \lambda \ge  \dfrac{3s+1}{3s},\\[1em]
	\dfrac{s}{6}&\textit{for}& \dfrac{s-1}{s}\leq \lambda \leq \dfrac{3s+1}{3s}, \\[1em]
	\dfrac{s^2}{4}\myp{\dfrac{3s-1}{3s}-\lambda}&\textit{for}& \lambda \le \dfrac{s-1}{s}.
\end{array}\right.
\]
\end{thm}
\begin{proof} We use the  same reasoning as in the proof of Theorem \ref{th_inverse_f}.
Applying relations \eqref{coffiecientK}, we obtain
\[
\abs{B_3-\lambda B_2^2}=\abs{d_3-(2-\lambda)d_2^2}.
\]
Assertion now follows by application of Theorem \ref{th.3.7} with $2-\lambda$. The inequalities are sharp for the functions
\begin{equation*}f(z)= \left\{
\begin{array}{ll}
        \overline{\mu} K_{s,2}(\mu z)\quad &\textit{for}\quad \dfrac{s-1}{s}< \lambda < \dfrac{3s+1}{3s},\\[0.5em]
       \overline{\mu} K_{s}(\mu z)\quad &\textit{for}\quad \lambda \in\myp{-\infty, \dfrac{s-1}{s}}\cup\myp{\dfrac{3s+1}{3s},\infty},\\[0.5em]
  \overline{\mu}F_x(\mu z)    \quad &\textit{for}\quad\lambda=\dfrac{s-1}{s},\\[0.5em]
   \overline{\mu}G_x(\mu z)    \quad &\textit{for}\quad\lambda = \dfrac{3s+1}{3s}.
\end{array}\right.
\end{equation*}
where $K_{s,2}$ and  $K_{s}$ are given by \eqref{K} and \eqref{K1},  $\mu$ is an unimodular constant, and  $F_x$ and $G_x$ $\myp{0\le x\le 1}$ are given by \eqref{fgx}.
\end{proof}

\noindent\textbf{Acknowledgements}

\noindent  The authors thank the editor and the anonymous referees for constructive and pertinent suggestions.\bigskip

%================================================================================================

% ------------------------------------------------------------------------
\end{document}